\documentclass[a4paper]{amsart}
\usepackage[utf8]{inputenc}
\usepackage[T1]{fontenc}
\usepackage[english]{babel}
\usepackage{tikz}
\usepackage{tikz-cd}
\usepackage{url}
\usepackage{wrapfig}
\usepackage{float}
\usepackage[margin = 1in]{geometry}
\usepackage{amsmath}
\usepackage{amssymb}
\usepackage{amsfonts}
\usepackage{amsthm}
\usepackage{mathtools}
\usepackage{enumerate}
\usepackage{verbatim}
\usepackage{physics}
\usepackage{euscript}
\usepackage{tabularx}
\definecolor{linkcolor}{rgb}{0.65,0,0}
\definecolor{citecolor}{rgb}{0,0.65,0}
\definecolor{urlcolor}{rgb}{0,0,0.65}
\usepackage[
            hidelinks,
            colorlinks=true,
            linkcolor=linkcolor,
            urlcolor=urlcolor,
            citecolor=citecolor,
        ]{hyperref}
\usepackage[ruled,vlined]{algorithm2e}
\usepackage{xstring}
\RequirePackage{mleftright} 
\RequirePackage{etoolbox}
\numberwithin{equation}{section}

\newcommand{\OO}{\mathcal{O}}

\newcommand{\CC}{\mathbb{C}}

\newcommand{\QQ}{\mathbb{Q}}
\newcommand{\ZZ}{\mathbb{Z}}
\newcommand{\FF}{\mathbb{F}}
\newcommand{\pp}{\mathfrak{p}}

\newcommand{\Pp}{\mathfrak{P}}

\DeclareMathOperator{\End}{End}

\DeclareMathOperator{\Gal}{Gal}

\DeclareMathOperator{\Hom}{Hom}
\DeclareMathOperator{\Aut}{Aut}

\DeclareMathOperator{\nrd}{nrd}

\DeclareMathOperator{\disc}{disc}

\DeclareMathOperator{\llog}{llog}

\DeclareMathOperator{\Spec}{Spec}

\DeclareMathOperator{\PGL}{PGL}
\newcommand{\M}{\mathsf{M}}

\newtheorem{thm}{Theorem}[section]
\newtheorem{prop}[thm]{Proposition}
\newtheorem{lem}[thm]{Lemma}

\theoremstyle{definition}
\newtheorem{defn}[thm]{Definition}
\newtheorem{rmk}[thm]{Remark}

\newtheorem*{rmk*}{Remark}
\newtheorem*{ex*}{Example}

\DeclarePairedDelimiter\ceil{\lceil}{\rceil}

\mleftright 

\DeclarePairedDelimiterX\thevect[1]{\mleft<}{\mright>}{
\ifblank{#1}{\:\cdot\:}{#1}
}

\newcommand\restr[2]{{
  \left.\kern-\nulldelimiterspace 
  #1 
  \vphantom{\big|} 
  \right|_{#2} 
  }}

\title{The SEA algorithm for endomorphisms of supersingular elliptic curves}

\author[Morrison]{Travis Morrison}
\address{Travis Morrison, Virginia Tech, USA}
\email{\href{mailto:tmo@vt.edu}{tmo@vt.edu}}
\thanks{}

\author[Panny]{Lorenz Panny}
\address{Lorenz Panny, Technische Universität München, Germany}
\email{\href{mailto:lorenz@yx7.cc}{lorenz@yx7.cc}}
\thanks{}

\author[Sot\'akov\'a]{Jana Sot\'akov\'a}
\address{Jana Sot\'akov\'a, University of Amsterdam, Netherlands}
\email{\href{mailto:ja.sotakova@gmail.com}{ja.sotakova@gmail.com}}
\thanks{}

\author[Wills]{Michael Wills}
\address{Michael Wills, University of Virginia, USA}
\email{\href{mailto:muf5cz@virginia.edu}{muf5cz@virginia.edu}}
\thanks{}

\date\today
\begin{document}

\begin{abstract}
    For a prime $p{\,>\,}3$ and a supersingular elliptic curve~$E$ defined over~$\FF_{p^2}$ with ${j(E)\notin\{0,1728\}}$, consider an endomorphism $\alpha$ of~$E$ represented as a composition of $L$ isogenies of degree at most $d$. We prove that the trace of $\alpha$ may be computed in  $O(n^4(\log n)^2 + dLn^3)$ bit operations, where $n{\,=\,}\log(p)$, using a generalization of the SEA algorithm for computing the trace of the Frobenius endomorphism of an ordinary elliptic curve. When $L\in O(\log p)$ and $d\in O(1)$, this complexity matches the heuristic complexity of the SEA algorithm. Our theorem is unconditional, unlike the complexity analysis of the SEA algorithm, since the kernel of an arbitrary isogeny of a supersingular elliptic curve is defined over an  extension of constant degree, independent of~$p$. We also provide practical speedups, including a fast algorithm to compute the trace of $\alpha$  modulo $p$.
\end{abstract}
\maketitle
\section{Introduction}
Let $p$ be a prime and let~$E$ be a supersingular elliptic curve defined over~$\FF_{p^2}$. An important subroutine in some algorithms for computing the endomorphism ring of~$E$ involves computing 
the trace of an endomorphism of~$E$. Indeed, the trace pairing 
on $\End(E)$ determines the structure of $\End(E)$ equipped with the quadratic form $\deg$ as a quadratic 
lattice up to isometry and hence its isomorphism class as an order in a quaternion algebra. More concretely, an efficient algorithm for computing the trace of an endomorphism of~$E$ yields an efficient algorithm for deciding whether $4$ endomorphisms are a $\ZZ$-basis for $\End(E)$: The endomorphisms $\alpha_1,\alpha_2,\alpha_3,\alpha_4$ span $\End(E)$ if and only if $\det(\tr(\alpha_i\alpha_j))=p^2$. 

We can compute the trace of $\alpha\in \End(E)$ using a generalization of Schoof's algorithm~\cite{SchoofSEA}, as observed by Kohel in his thesis~\cite[Theorem~81]{Koh96}. We will assume the degree 
of $\alpha$ is smooth and thus $\alpha$ may be represented as a 
sequence of small degree isogenies. To compute $\tr\alpha$, we compute $\tr\alpha\bmod{\ell}$ for enough small primes $\ell\leq B$ such that $\prod_{\ell<B} \ell >4\sqrt{\deg\alpha}$.  
We can then recover $\tr \alpha$ using the Chinese remainder theorem. For a prime $\ell$, if $E[\ell]\subseteq \ker\alpha$, then $\tr \alpha \equiv 0 \pmod{\ell}$. Otherwise,  there exists a $P\in E[\ell]$ such that $\alpha(P)$ is nonzero, and to compute $t_{\ell} = \tr \alpha\bmod{\ell}$ it suffices to find a $0\leq c<\ell$ such that  and 
$\alpha^2(P)+[\deg\alpha]P = c\alpha(P)$.
If such a $P$ exists, then $t_\ell = c$. And to check whether such a $P$ exists, we can 
compute the restrictions of $\alpha^2+[\deg\alpha]$ and $[c]\circ \alpha$ to $E[\ell]$. This involves computing the coordinate functions of $\restr{\alpha}{E[\ell]}$ and $\restr{[1]}{E[\ell]}$ and then performing arithmetic in $\End(E[\ell])$ using the group law on~$E$ to compute 
\[
    \restr{\alpha^2}{E[\ell]} = \big(\restr{\alpha}{E[\ell]}\big)^2,\quad \restr{[d]}{E[\ell]} = d\restr{[1]}{E[\ell]},\quad \restr{(c\alpha)}{E[\ell]}=c\big(\restr{\alpha}{E[\ell]}\big),
\]
\[\text{and }\restr{\big(\alpha^2+[d]\big)}{E[\ell]} = \restr{\alpha^2}{E[\ell]}+\restr{[d]}{E[\ell]}.\]
These restrictions can be represented as $(a(x),b(x)y)$ where $a,b\in \FF_{q}[x]/(\psi_{\ell})$ where $\psi_{\ell}$ is the $\ell$th division polynomial of~$E$. 

To speed up this computation, which involves polynomial arithmetic modulo the degree $(\ell^2-1)/2$ polynomial $\psi_{\ell}$, we might look for inspiration in the SEA algorithm, and specifically to Elkies' improvement to Schoof's algorithm, for computing the trace of the Frobenius endomorphism~$\pi_E$ of an ordinary curve~$E$. Elkies' idea is to note that, when an odd prime $\ell$ splits in $\ZZ[\pi_E]$,~$E$ admits a rational $\ell$-isogeny $\phi\colon E\to E'$ and thus  the division polynomial  $\psi_{\ell}$ has a nontrivial factor $h\in \FF_q[x]$. Indeed, the polynomial $h$ is the {\em kernel polynomial} of $\phi$ and has degree $(\ell-1)/2$. Such a prime is called an {\em Elkies prime} for~$E$.  Asymptotically, 50\% of the primes are Elkies primes for~$E$. Since $\ker\phi$ is defined over~$\FF_{q}$, it is invariant under $\pi_E$ so $\restr{\pi_E}{E[\ell]}$ has a nontrivial eigenvalue $\lambda \in \ZZ/\ell\ZZ$. Then $t_{\ell} \equiv \lambda + q/\lambda\pmod{\ell}$. Since $\deg h$ is smaller than $\deg \psi_{\ell}$ by a factor on the order of $\ell$, computing $\restr{\pi_E}{\ker\phi}$ is about $\ell$ times faster than computing $\restr{\pi_E}{E[\ell]}$. This gives a linear speedup in computing $t_{\ell}$ whenever $\ell$ is an Elkies prime. Primes $\ell$ that are not Elkies primes are called Atkin primes due to an improvement of Atkin for speeding up the computation of $t_{\ell}$. Atkin's method does not give an asymptotic speedup, and not even the Generalized Riemann Hypothesis is strong enough to prove that there are enough small Elkies primes for~$E$ so that the SEA algorithm is asymptotically faster than Schoof's algorithm, as shown by Shparlinski and Sutherland~\cite[Corollary~15]{ShSureductions}. In practice, there are enough Elkies primes and the SEA algorithm performs much better than one is able to prove. Atkin primes are skipped entirely in Sutherland's point-counting records using the SEA algorithm~\cite{sutherlandrecords} together with his algorithm for computing {\em instantiated modular polynomials}~\cite{SutherlandEvaluation}. Finally, Shparlinski and Sutherland~\cite{ShSudistribution} show that when one averages over elliptic curves over $\FF_q$, there are as many Elkies primes as Atkin primes $\ell<L$ so long as $L\geq (\log p)^{\epsilon}$ for fixed $\epsilon>0$ and sufficiently large $q$, and in another direction, the same authors show in ~\cite{ShSureductions} that for a fixed elliptic curve~$E$ over $\QQ$ and for almost all primes $p$ of good reduction for~$E$, the SEA algorithm is faster than Schoof's algorithm for computing the number of points of the reduction of~$E$ modulo $p$.

Returning to our setup, and with Elkies' improvement to Schoof's algorithm and the above discussion in mind,  we now leverage the assumption that~$E$ is supersingular and defined over~$\FF_{p^2}$ and that ${j(E)\notin\{0,1728\}}$. Then every isogeny of~$E$ defined over~$\overline{\FF_{p^2}}$ is actually defined over~$\FF_{p^2}$; we prove a more general version of this in Proposition~\ref{prop:isogeny_field_def}. Therefore every prime $\ell$ is an Elkies prime for~$E$! Let $\alpha$ be an arbitrary endomorphism of~$E$ and let $h\in \FF_{p^2}[x]$ be the kernel polynomial of an $\ell$-isogeny $\psi$ of~$E$. There is no reason for $\ker\psi$ to be stable under $\alpha$, i.e., $\alpha$ is not necessarily an endomorphism of $\ker\psi$, unlike the Frobenius endomorphism. However, we can still compute restrictions of $c\alpha,\alpha^2,[\deg\alpha]$, add them together, and test for equality using arithmetic in the ring $R_\psi=\FF_{p^2}[x,y]/(y^2-f(x),h(x))$. One can interpret this as performing arithmetic in $E(R_\psi)$, the $R_\psi$-points in the group scheme determined by $E$;  see Section~\ref{subsec:mod_ell}.   
This allows us to compute $t_{\ell}\coloneqq \tr\alpha\bmod{\ell}$ as in the SEA algorithm. We show in Theorem~\ref{thm:TraceModEll} that when $\alpha$ is the composition of $L\in O(\log p)$ isogenies of degree $d\in O(1)$ and $\ell\in O(\log p)$ that $t_{\ell}$ can be computed with an expected $O(n^3(\log n)^3)$ bit operations, where $n=\ceil{\log p}$. In Theorem~\ref{thm:cycletrace}, we show that $\tr\alpha$ can be computed with an expected $O(n^4(\log n)^2)$ bit operations. 

First, we stress that this result is unconditional and matches the heuristic complexity analysis of the SEA algorithm: as discussed above, heuristics beyond the Generalized Riemann Hypothesis are required to prove that the SEA algorithm is asymptotically faster than Schoof's algorithm~\cite[pp.~2--3]{ShSureductions}. Next, we observe that under our assumption on $\alpha$, the input has size $O((\log p)^2)$, so this is actually a quasi-quadratic algorithm. Finally, our assumption that the separable degree of $\alpha$ is smooth of size $p^{O(1)}$ is reasonable since the best algorithms for generating endomorphisms of a supersingular elliptic curve will satisfy this assumption~\cite{DG16, EHLMP20, fuselier2023computingsupersingularendomorphismrings}.

For another work involving computing the trace of an endomorphism of~$E$ other than the Frobenius endomorphism, see~\cite{Qcurvetraces}, in which Morain--Scribot--Smith compute point counts on curves $E/\FF_{p^2}$ with a $d$-isogeny $\phi \colon E\to E^{(p)}$ by computing the trace of $\psi=\pi\circ\phi$ modulo Elkies primes. They observe that for an Elkies prime $\ell$ such that~$E$ has precisely two rational $\ell$-isogenies with kernels $K_1$ and $K_2$, each kernel $K_i$ is invariant under $\psi$. For another related work, see~\cite{Schoofcycles}, where Couveignes--Morain use the existence of an $\ell^n$-isogeny cycle in a modified SEA algorithm to compute the trace of Frobenius modulo $\ell^n$ (while we are modifying the SEA algorithm to compute the trace of a cycle!). 

\subsection*{Outline} In Section~\ref{sec:background}, we prove that the kernel of every isogeny of a supersingular elliptic curve~$E$ over $\FF_q$ is defined over an extension of degree at most $3$. This result implies we may  use Elkies' method to compute the trace of an endomorphism of a supersingular curve modulo $\ell$ for every odd prime $\ell$. In Section~\ref{sec:tracemod}, we give algorithms for computing the trace of a supersingular endomorphism modulo a prime. In \ref{subsec:mod_ell} we give our algorithm for computing the trace of an endomorphism $\alpha$ modulo an odd prime $\ell\neq p$ by restricting to the kernel $K$ of a rational $\ell$-isogeny, even when $\alpha$ does not fix $K$, and analyze its complexity. We discuss how arithmetic with restrictions of endomorphisms to kernels is performed, i.e., how to represent, add together, evaluate under isogenies, and test equality of elements in $E(R_{\psi})$. The results and analysis in this section apply to any elliptic curve over a finite field of characteristic $p>3$ with a rational isogeny of degree $\ell$.  In \ref{subsec:mod_p} we show how to compute the trace of a separable endomorphism modulo $p$, the characteristic, using the action of endomorphisms on invariant differentials. In~\ref{sec:trace} we prove our main Theorem, Theorem~\ref{thm:cycletrace}, on the complexity of the SEA algorithm for supersingular endomorphisms. We discuss some timings of our implementation\footnote{\url{https://github.com/travismo/beyond-the-SEA/}} in Section~\ref{sec:timings} and conclude in Section~\ref{sec:dewaghe} with a discussion of how the ideas in Section~\ref{subsec:mod_ell} might apply to the SEA algorithm for computing the trace of Frobenius modulo Atkin primes, an idea already explored by Dewaghe~\cite{Dew98}. 

\subsection*{Acknowledgements} We thank Fran\c{c}ois Morain who informed us that the idea of applying Elkies' algorithm for Atkin primes $\ell$ for an ordinary curve~$E$ that a rational $\ell$-isogeny over a small but nontrivial  extension of $\FF_q$ is due to Dewaghe~\cite{Dew98}. We also thank Andrew Sutherland for helpful discussions, and the Park City Mathematics Institute, where this work began. Finally, we are grateful to the two reviewers for their careful reading of this paper. We are especially grateful to one reviewer whose suggestions significantly simplified Section~\ref{subsec:mod_ell} and who suggested much simpler proof of Lemma~\ref{lem:normalizing_constant_from_numerator}. This material is based upon work supported by the National Science Foundation under Grant No.\ 2340564. This research was funded in part by the Commonwealth of Virginia’s Commonwealth Cyber Initiative (CCI). The first author acknowledges support of the Institut Henri Poincaré (UAR 839 CNRS-Sorbonne Université), and LabEx CARMIN (ANR-10-LABX-59-01).
\section{Background}\label{sec:background}
In this section, we fix notation and cover some of the relevant background on computational aspects of elliptic curves over finite fields. 

\subsection{Isogenies of elliptic curves}
Let $p>3$ be a prime, let $q=p^n$, and let~$E$ be an elliptic curve defined over~$\FF_q$. Throughout, we will assume $p>3$, so~$E$ is 
given by a short Weierstrass equation $E\colon y^2=x^3+Ax+B$ where $A,B\in \FF_q$ and $4A^3+27B^2\neq 0$. If $E'$ is another elliptic curve, an {\em isogeny} $E\to E'$ is a non-constant
rational map that maps the identity on~$E$ to the identity on $E'$. 
If $\phi$ satisfies $\pi^n\circ\phi=\phi\circ\pi^n$, where $\pi$ is the $q$\nobreakdash-power Frobenius endomorphism,
we will say $\phi$ is an isogeny of~$E$ defined over~$\FF_{q^n}$. The {\em degree} of an isogeny $\phi$ is its degree as a rational map, and $\phi$ is separable if it is separable as a rational map. If $\phi$ is separable, then $\deg\phi = \#\ker\phi$. For every isogeny $\phi\colon E\to E'$, there is a unique {\em dual} isogeny $\widehat{\phi}\colon E'\to E$ such that $\widehat{\phi}\circ \phi = [\deg\phi]$ and $\phi\circ\widehat\phi=[\deg\phi]$. An {\em endomorphism} of~$E$ is either an isogeny or the constant $0$-map $[0]$ from~$E$ to~$E$. The degree map is a positive-definite quadratic form on $\End(E)$. The {\em trace} of an endomorphism $\alpha\in \End(E)$ is the integer $\tr\alpha$ such that $\alpha+\widehat{\alpha} = [\tr \alpha]$, and the bilinear pairing on $\End(E)$ corresponding to the degree map is
\(
(\alpha,\beta)\mapsto \tr \alpha\widehat{\beta}
\).

For purposes of calculation and algorithms, we will assume our isogenies are in {\em standard form}:
\begin{defn}
    Let~$E$ and $E'$ be elliptic curves in short Weierstrass form and let~$E$ be given by $y^2=x^3+Ax+B$. 
    Let $\phi\colon E\to E'$ be a separable $\ell$-isogeny. We say that $\phi$ is in {\bf standard form} if 
    \[
    \phi(x,y)=\bigg(\frac{u(x)}{v(x)},cy\frac{s(x)}{t(x)}\bigg)
    \]for polynomials $u,v,s,t\in k[x]$ where
    \begin{enumerate}
    \item $v(x)$ is the polynomial
    \[
        v(x) = \prod_{\substack{Q\in \ker \phi\\Q\neq0}}(x-x_Q),
    \]
    \item  $s(x)/t(x)=(u(x)/v(x))'$.
    \end{enumerate}
\end{defn}

By Proposition 4.1 of~\cite{BMSS}, every separable isogeny may be written in standard form. Suppose~$E$ is the curve $y^2=x^3+Ax+B$, that  $\phi$ is separable and has degree $\ell$, and let $-\sigma$ be the coefficient of $x^{\ell-1}$ in $v(x)$. By Proposition 4.1 of~\cite{BMSS}, if $c=1$  the polynomial $u(x)$ satisfies the relation 
\[
\frac{u(x)}{v(x)} = \ell x - \sigma -(3x^2+A)\frac{v'(x)}{v(x)}-2(x^3+Ax+B)\Big(\frac{v'(x)}{v(x)}\Big)'.
\]

The {\em kernel polynomial} of an isogeny $\phi$ is the monic polynomial that vanishes precisely once at the collection of $x$-coordinates of the nonzero points in $\ker\phi$. If $h(x)$ is the kernel polynomial of $\phi(x)$ and~$E$ is given by $y^2=f(x)$, then $v(x)$ above is 
\[
v(x) = \frac{h(x)^2}{\gcd(f(x),h(x))}.
\]
In particular, when $\phi$ has odd degree, we have $h(x)^2=v(x)$.
The isogeny $\phi$ is {\em normalized} if $\phi^*\omega' = \omega$ where $\omega,\omega'$ are the invariant differentials $dx/2y$ on~$E$ and $E'$ respectively. Equivalently, $\phi$ is normalized if the constant $c$ in the standard form of $\phi$ is $1$. 
A normalized isogeny $\phi$ is defined over~$\FF_q$ if and only if $\ker\phi$ is $\Gal(\overline{\FF_q})$-stable, and this holds if and only if $h(x)\in \FF_q[x]$. 
Since $v(x)$ is determined by the kernel of $\phi$, if $\phi$ is normalized the polynomial $u(x)$ (and hence the polynomials $s$ and $t$) are completely determined by $\ker\phi$ and the equation $y^2=x^3+Ax+B$ defining~$E$.  Every finite subgroup $K$ of $E(\overline{\FF_q})$ determines an isogeny $E\to E/K$; here $E/K$ is the unique elliptic curve $E'$ such that the separable isogeny $E\to E'$ with kernel $K$ is normalized. While we do not need it, there are formulas for the coefficients giving a Weierstrass equation defining $E/K$ from the data~$E$ and $K$; see the discussion following \cite[Proposition~4.1]{BMSS} in Section~4.1. The constant $c$ in the standard form results from post-composition of the normalized isogeny $E\to E/K$ with the isomorphism $(x,y)\mapsto (c^{-2}x,c^{-3}y)$. 

\subsection{Supersingular elliptic curves}
An elliptic curve~$E$ defined over~$\FF_q$ is {\em supersingular} if its geometric endomorphism algebra is a quaternion algebra. Equivalently, the {\em trace of Frobenius} $t_E\coloneqq q+1-\#E(\FF_q)$ is congruent to $0$ modulo $p$. This restricts the possibilities for how $\pi_E$ can act on the $\ell$-torsion of~$E$. This, in turn, restricts the minimal field of definition of an isogeny between two supersingular elliptic curves. 
Every supersingular elliptic curve~$E$ over $\FF_{p^2}$ with $j(E)\notin\{0,1728\}$ has all its isogenies defined over~$\FF_{p^2}$. We'll now show more generally that a supersingular curve over $\FF_q$ has all its isogenies defined over an extension of degree at most $3$. 
\begin{lem}\label{prop:isogeny_field_def}
    Let $q=p^e$ be a power of a prime $p>3$ and let~$E$ be a supersingular elliptic curve defined over~$\FF_q$.  Then if $\psi\colon E\to E'$ is any isogeny of~$E$, its kernel is defined over the extension $\FF_{q^m}$ with $m=1,2,$ or $3$. If $j(E)\notin\{0,1728\}$, then we can take $m=1$ if $e$ is even and $m=2$ if $e$ is odd. 
\end{lem}

\begin{proof}
    According to \cite[Theorem~4.2(ii) and~(iii)]{Schoofnonsingularplanecubics},
    the $q$\nobreakdash-power Frobenius $\pi$ of $E$ satisfies a polynomial
    equation $\pi^2-k\sqrt{q}\pi+q=0$, where $k=0$ for odd $e$ and $k\in\{0,\pm1,\pm2\}$ for even $e$.
    A calculation shows that $\pi^4-(k^2-2)q\pi^2+q^2=0$ and $\pi^6-(k^2-3)kq^{3/2}\pi^3+q^3=0$.
    \begin{itemize}
        \item If $k=0$, then $0=\pi^2+q$, hence $\pi^2=-q\in\ZZ$. Use $m=2$.
        \item If $k\in\{\pm1\}$, then $0=\pi^6\pm2q\sqrt{q}\pi^3+q^3=(\pi^3\pm q\sqrt{q})^2$, hence $\pi^3=\mp q\sqrt{q}\in\ZZ$. Use $m=3$.
        \item If $k\in\{\pm2\}$, then $0=\pi^4-2q\pi^2+q^2=(\pi^2-q)^2$, hence $\pi^2=q\in\ZZ$. Use $m=2$.
    \end{itemize}
    Therefore, for all respective exponents~$m$, we have~$\pi^m\in\ZZ$ and thus~$\pi^m(\ker \psi)=\ker \psi$.

    Finally,
    for $j(E)\notin\{0,1728\}$,
    set~$m=1$ if~$e$ is even
    and~$m=2$ if~$e$ is odd.
    Then $\sqrt{q}^m\in\ZZ$.
    Since  $\pi^m$ and $[\sqrt{q}^m]$
    are both purely inseparable endomorphisms of degree $q^m$,
    there
    exists an automorphism $\zeta\in\Aut(E)$
    such that $\pi^m=\zeta[\sqrt{q}^m]$.
    But  the assumption $j(E)\notin\{0,1728\}$
    implies $\Aut(E)=\{[\pm1]\}$,
    so $\pi^m\in\ZZ$
    and thus $\pi^m(\ker\psi)=\ker\psi$.
\end{proof}

\noindent
In particular,
Lemma~\ref{prop:isogeny_field_def}
reveals that every isogeny
between supersingular elliptic
curves over $\FF_q$ is defined
over~$\FF_{q^6}$:
This is the smallest extension of $\FF_q$
containing both $\FF_{q^2}$ and $\FF_{q^3}$.

\subsection{Arithmetic in \texorpdfstring{$\FF_q[x]/(h(x))$}{Fq[x]/h(x)}}
Just as in Schoof's algorithm, we will reduce computing the trace of an endomorphism of an elliptic curve~$E$ to various algebraic operations over finite fields. For details on the complexity of arithmetic over finite fields, see~\cite[\S\,8,11]{moderncomputeralgebra} and for an introductory overview, see~\cite[Lecture~3]{MITEC}. Let $\M(n)$ denote the number of bit operations required to multiply two $n$-bit integers.  Let $q=p^e$ be a power of an odd prime and let $n= \ceil{\log q}$. Elements of $\FF_q$ can be multiplied by lifting to polynomials of $\FF_p[x]$ of degree $e-1$ and then multiplied with Kronecker substitution~\cite[\S\,8.4]{moderncomputeralgebra} requiring $O(\M(e(\log p + \log e)))$ bit operations. This cost simplifies to $O(M(e\log p))=O(\M(n))$ under the assumption that $\log e \in O(\log p)$. We will assume $\log e =O(\log p)$ throughout; in our particular case of interest, we have $e=2$, since a supersingular elliptic curve has a model defined over $\FF_{p^2}$. We can compute an inverse of an element in $\FF_q^{\times}$ with $O(\M(n)\log n)$ bit operations using fast Euclidean division~\cite[\S11.1]{moderncomputeralgebra} and Kronecker substitution. Let $f,g,h$ be polynomials in $\FF_q[x]$ of degree at most $d\in O(\log n)$ and let $R=\FF_q[x]/(h(x))$. The product $fg\bmod{h}$ can be computed with $O(\M(dn))$ bit operations  Kronecker substitution~\cite[Corollary~8.28]{moderncomputeralgebra}. We can test if $f\in R^{\times}$, i.e., if $\gcd(f,h)=1$, and if so compute $f^{-1}\bmod{h}$ with $O(\M(dn)\log(dn))$ bit operations using fast Euclidean division and Kronecker substitution. A root of $f$ in $\FF_q$ can be found with Rabin's probabilistic algorithm~\cite{Rabin80} using $O(n\M(dn))$ bit operations, again using Kronecker substitution and our assumption that $\log d \in O(n)$. The bound $\M(n)\in O(n\log n)$ holds by~\cite{vdHH21}. We summarize this in the following table which will be used in the complexity analysis in Section~\ref{subsec:mod_ell}. 

\begin{prop}\label{prop:basic_complexity_table}
    Let $p$ be an odd prime, let $\log e \in O(\log p)$ be an integer, let $q=p^e$, and set $n\coloneqq \ceil{\log q}$. Let $d$ be a positive integer and assume $d\in O(\log n)$. Let $a,b\in \FF_q^{\times}$ and let $f,g,h\in \FF_q[x]$ be polynomials of degree at most $d$ and assume $\log d \in O(\log n)$. The  following table gives the bit complexity of computing products and inverses of elements of $\FF_q[x]$ and $\FF_q[x]/(h(x))$ as well as root-finding in $\FF_q$.
    \begin{center}
    \begin{tabular}{c|c}
        Operation & bit complexity \\ \hline
        $ab$ & $O(n\log n)$ \\ 
        $a^{-1}$ & $O(n(\log n)^2)$ \\
        $fg\bmod{h}$ & $O(dn\log(dn))$ \\ 
        $f^{-1}\bmod{h}$ & $O(dn(\log(dn))^2)$ \\
       $f(r)=0$  & $O(n^2\log n)$ \\
    \end{tabular}        
    \end{center}

\end{prop}

\section{Computing the trace modulo a prime}\label{sec:tracemod}
Suppose $\alpha$ is a non-integer endomorphism of an elliptic curve~$E$ defined over~$\FF_q$. Schoof's algorithm, originally designed to compute the trace of the Frobenius endomorphism  $\pi_E$,   computes  $\tr\alpha$ by computing $t_{\ell}\coloneqq \tr\alpha\bmod{\ell}$ for enough primes $\ell\in O(\log p)$ in order to recover $\tr\alpha$ with the Chinese Remainder Theorem. This is possible since $\ZZ[\alpha]$ is an imaginary quadratic order, so we have the bound $\lvert\tr\alpha\rvert < 2(\deg\alpha)^{1/2}$. Suppose now~$E$ is supersingular, defined over~$\FF_{p^2}$, and that $j(E)\notin\{0,1728\}$. Then
\[
    t_E \coloneqq  p^2+1-\#E(\FF_q)=\pm 2p
\]
is the trace of $\pi_E$, and $\pi_E$ is $[p]$ or $[-p]$. This implies every isogeny of~$E$ is defined over~$\FF_{p^2}$ (see~\ref{prop:isogeny_field_def}); for $E$, every odd prime is an Elkies prime! We prove in Theorem~\ref{thm:TraceModEll} that $t_{\ell}$ may be computed in $O(n^3(\log n)^3)$ time, assuming $\log \deg\alpha \in O(n)$ and $\alpha$ is represented by a sequence of isogenies of degree bounded by a constant. Note that this result is {\em unconditional}\,---\,we do not require GRH since all primes are Elkies primes. The method is discussed in Section~\ref{subsec:mod_ell}.

Another advantage in our case is that we can actually compute the trace of $\alpha$ modulo $p$. This does not impact the asymptotic complexity but gives a dramatic speedup in practice. We discuss this in Section~\ref{subsec:mod_p}. Another practical speedup is that $\#E(\FF_{p^2})$ is easily computed since we assume~$E$ is supersingular. Thus whenever $\ell$ divides $\#E(\FF_{p^2})$, we can compute $t_{\ell}$ working directly with a rational point of order $\ell$. Again, this does not provide an asymptotic speedup, but the idea is helpful in practice.

\subsection{Computing the trace modulo  an odd prime \texorpdfstring{$\ell\not=p$}{l!=p}}
\label{subsec:mod_ell}
Let $E$ be an elliptic curve over $k=\FF_q$, a finite field of characteristic $p>3$, given in short Weierstrass form $E:y^2=f(x)=x^3+Ax+B$ for some $A,B\in k$. Let $\alpha\in \End(E)$. In this section we will describe how to compute $t_\ell=\tr\alpha\pmod{\ell}$ for a fixed prime $\ell\not\in\{ 2,p\}$; we will generally think of $\ell$ as being of size $O(\log p)$. Recall that $\alpha$ satisfies the characteristic polynomial $\alpha^2 + t\alpha + [d] = 0$ in $\End(E)$, where $t = \tr(\alpha)$ and $d = \deg \alpha$. Endomorphisms of $E$ induce group endomorphisms of $E(\bar{k})$ by evaluation, which suggests the following procedure for computing $t_\ell$. First, find some $P\in E[\ell]\subseteq E(\bar{k})$ of exact order $\ell$. Note that since $\ell$ does not divide $\deg\alpha$, it must be that $\alpha(P)$ is of exact order $\ell$ as well (and in particular, $P\not\in \ker \alpha$). Then, compute $\alpha^2(P) + [d]P$ and $c\alpha(P)$ for $0\leq c < \ell$. Vanishing of the characteristic polynomial implies that 
\[
    \alpha^2(P) + [d]P = c\alpha(P) \quad \Longleftrightarrow \quad (t-c)\alpha(P) = 0.
\]
Since $\alpha(P)$ is of exact order $\ell$, the above holds if and only if $\ell \mid (t-c)$; that is, when $c = t_\ell$.

Let us generalize this. We may regard $E$ as a group scheme over $k$, representing the functor $\Hom(-,E)$ which associates to any $k$-scheme $S$ an abelian group structure on $E(S) := \Hom_k(S,E)$. By functoriality, endomorphisms of $E$ induce group endomorphisms on $E(S)$ via post-composition. Thus, the procedure described in the preceding paragraph may be carried out for $P\in E(S)$ of order exactly $\ell$ for any $S/k$, not just $\Spec \bar{k}$. Throughout this section, we will describe schemes $S$ and points $P\in E(S)$ that can be used to efficiently compute $t_\ell$. In particular, we will leverage the fact that all supersingular elliptic curves over a finite field of characteristic $p$ have $j$-invariant in $\mathbb{F}_{p^2}$ to efficiently compute an $\ell$-isogeny $\psi:E\to E/\ker\psi$, then use $S = \ker\psi- \{0_E\}$.

For a general $k$-scheme $S$, the group $E(S)$ may be difficult to work with; algorithms for point addition or checking equality of points may not be readily available. In the affine case, however, things are slightly better behaved. In particular, let $S = \Spec R$; we will denote $E(R) = E(\Spec R)$. Consider the usual affine open subscheme $E^\circ  = E- \{0_E\}\subseteq E$, given explicitly by $E^\circ = \Spec k[E]$ with
\[
    k[E] = \frac{k[x,y]}{\big(y^2 - f(x)\big)}.
\]
We similarly let $E^\circ(R) = E^\circ(\Spec R)$, which is the set consisting of all regular morphisms $\Spec R \to E^\circ$. We may regard $E^\circ(R)$ as the subset of $E(R)$ consisting of those morphisms which do not have $0_E$ in their image. Note that $E^\circ(R)$ is not a group, and that the group operation is not closed on this subset. By the universal property of the $\Spec$ functor, points in $E^\circ(R)$ are in bijection with $k$-algebra homomorphisms $\phi:k[E]\to R$. Under the association $\phi\mapsto (\phi(x), \phi(y))\in R^2$, these are in bijection with pairs $(a,b)\in R^2$ satisfying $b^2 = f(a)$. Thus, the problem of checking equality in $E^\circ(R)$ is reduced to that of checking equality in $R$.

\begin{rmk}
    The above is exactly the process by which the scheme-theoretic points of $E^\circ(k)$ correspond to the usual tuples $(a,b) \in k^2$ used to represent them.
\end{rmk}

We now turn our attention to the problem of the group law on $E(R)$, which is generally more subtle than when working over a field. We will focus on the case in which $R$ is an \'etale $k$-algebra. Here, we may fix an isomorphism $R \cong L_1\times \cdots \times L_r$ for some finite separable extensions $L_i/k$. Since $E$ is a commutative group scheme, we obtain an isomorphism of abelian groups
\[
    E(R) \cong E(L_1\times \cdots \times L_r) \cong E(L_1)\oplus \cdots \oplus E(L_r).
\]
We may thus regard points in $E(R)$ as ordered tuples $(\phi_1,\hdots, \phi_r)$ of $k$-morphisms $\phi_i:\Spec L_i \to E$. After base change to the algebraic closure, the maps $\phi_i$ themselves correspond to $[L_i:k]$-many points in $E(\bar{k})$ which are fixed by $\Gal(\bar{k}/L_i)$ and can be partitioned into some number of $\Gal(\bar{k}/k)$-orbits. In this way, a point in $E(R)$ corresponds to $\dim_k R$ many points of $E(\bar{k})$, perhaps with repetition. Since a map $R\to E$ has image containing $0_E$ exactly when one of its coordinates is the point $0_E\in E(L_i)$, we obtain a bijection of sets
\[
    E^\circ(R) \cong E^\circ(L_1)\times \cdots \times E^\circ(L_r).
\]
This also highlights the difficulty in stating a group law uniformly on $E(R)$; the group law for $E(L)$ with $L/k$ a field is already defined in multiple cases, and having $r$ different coordinates only exacerbates this problem. However, in many cases we may use the familiar chord-and-tangent formulas (e.g.~as in \cite{aec}) to perform arithmetic on $E(R)$.

\begin{prop}\label{prop:the_usual_algorithms_work}
    Let $R$ be an \'etale $k$-algebra, and suppose we have points $P,P'\in E^\circ(R)$. Write $P=(a,b)$ and $P'=(a',b')$
    \begin{enumerate}
        \item If $a-a' \in R^\times$, then $P+P' \in E^\circ(R)$ and has coordinates $(a'', b'')$ with
        \begin{equation}\label{eqn:chord_addition_formula}
            a'' = \left(\frac{b' - b}{a' - a}\right)^2 - a - a' \quad \text{and} \quad b'' = \left(\frac{b' - b}{a' - a}\right)(a - a'') - b.
        \end{equation}
        
        \item If $b\in R^\times$, then $2P\in E^\circ(R)$ and has coordinates $(a'',b'')$ with
        \begin{equation}\label{eqn:tangent_addition_formula}
            a'' = \left(\frac{3a^2 + A}{2b}\right)^2 - 2a \quad \text{and} \quad b'' = \left(\frac{3a^2 + A}{2b}\right) (a - a'') - b.
        \end{equation}

        \item Let $\phi:E\to E'$ be a separable isogeny given in standard form as $(\frac{u(x)}{v(x)}, \frac{s(x)}{t(x)}y)$. If $v(a) \in R^\times$ (or equivalently $t(a)\in R^\times$), then $\phi(P)\in E^\circ(R)$ and has coordinates $\left(\frac{u(a)}{v(a)}, \frac{s(a)}{t(a)}b\right)$.
    \end{enumerate}
\end{prop}
\begin{proof}
    Suppose without loss of generality that $R = k[x]/\big(h_1(x)\cdots h_r(x)\big)$ with each $h_i(x)$ irreducible. Then taking $L_i = k[x]/(h_i(x))$, the isomorphism $R\cong L_1\times \cdots \times L_r$ is given by
    \[
        a \mapsto \big( a\pmod{h_1}, ~\hdots,~ a\pmod{h_r}\big).
    \]
    This then gives an explicit description of the isomorphism $E(R)\cong \bigoplus_i E(L_i)$ on $E^\circ(R)$; it is the map $(a,v) \mapsto \big(a\pmod{h_i}, ~b\pmod{h_i}\big)_{i=1}^r$. Let $P_1,\hdots, P_r$ and $P_1',\hdots, P_r'$ denote the images of $P$ and $P'$ respectively in the $E(L_i)$.
    
    We will prove (1); the proofs of (2) and (3) follow similarly. The condition that $a-a'\in R^{\times}$ is equivalent to $a-a'\not=0\pmod{h_i}$ for all $i$, which is to say that $a-a'$ is non-zero in each of the $L_i$. It follows that the addition formula in each of the $E(L_i)$ is given by (\ref{eqn:chord_addition_formula}). Letting $c$ denote the inverse of $a-a'$ in $R$, we can take the representative of $(a-a')^{-1} \pmod{h_i}$ in each coordinate to be the image of $c$. In this way, we obtain $a'',b''\in R$ such that
    \[
        P_i+P_i' = \big(a''\pmod{h_i},~ b''\pmod{h_i}\big)
    \]
    for every $i$. Then $P+P'$ must also have coordinates $(a'',b'')$ as well, which are defined by the formulas (\ref{eqn:chord_addition_formula}) as desired. \qedhere

\end{proof}

\begin{rmk}
    If the conditions laid out in Proposition \ref{prop:the_usual_algorithms_work} part (1) are not met, this does not necessarily mean that $P+P'\in E(R) - E^\circ(R)$. For example, it may be that $P$ and $P'$ have the same value on some $E(L_i)$ coordinate (and there would need the doubling formula) but not the others, so that there is not a unified formula that computes $P+P'$ in all coordinates.
\end{rmk}

The following lemma will also be useful when we need to perform computations with prime-order points of $E^\circ(R)$.

\begin{lem}\label{lem:ell_torsion_nice_under_coprime_isogeny}
    Let $R$ an \'etale $k$-algebra, and let $\phi:E\to E'$ an isogeny. Suppose that none of the coordinates of some point $P\in E^\circ(R)$ lie in $\ker \phi$. Then writing $\phi = (\frac{u(x)}{v(x)}, \frac{s(x)}{t(x)}y)$ in standard form, we have that $v(x(P))\in R^\times$. It follows from Proposition \ref{prop:the_usual_algorithms_work} that $\phi(P) \in (E')^\circ(R)$.
   
\end{lem}
\begin{proof}
    Again, we decompose $E(R) \cong \bigoplus_{i=1}^r E(L_i)$ and let $P = (P_1,\hdots, P_r)$ under this isomorphism. 
    Suppose that $v(x(P))\not\in R^\times$. This implies that $v(x(P_i)) = 0$ in $L_i$ for some $i$, which in turn means that $P_i\in \ker\phi$, now regarding $\phi$ as a homomorphism $E(L_i)\to E'(L_i)$. Then $P$ has a coordinate contained in $\ker\psi$. 
   
\end{proof}

We will now describe the points we wish to use to compute $t_\ell$.
Suppose we have an $\ell$-isogeny $\psi:E\to E/\ker\psi$ with kernel polynomial $h(x)$. By construction, the finite scheme $\ker\psi-\{0_E\}$ is given by $\Spec R_\psi$ with
\[
    R_\psi = \frac{k[x,y]}{\big(h(x), y^2-f(x)\big)}.
\]
This scheme admits a canonical inclusion to $E^\circ$ (given by the quotient map on rings), and we will denote by $P_\psi\in E(R_\psi)$ the corresponding point. The $k$-algebra $R_\psi$ is \'etale of $k$-dimension $\ell-1$, and the point $P_\psi$ corresponds to the $\ell-1$ points of $\ker \psi - \{0_E\}$ in $E(\bar{k})$. Since $\psi$ is an $\ell$-isogeny, these points all individually have order $\ell$, so $P_\psi$ is of order $\ell$ in $E(R_\psi)$. Further, since any of these $\bar{k}$-points corresponding to $P_\psi$ generate all of the others via rational maps, we may fix a decomposition $R_\psi \cong \prod_{i=1}^r L$ into $r$ copies of the same field $L/k$, where $r\cdot [L:k] = \ell - 1$. 

Heuristically, one expects any $r\mid (\ell-1)$ to be possible for various choices of $E$ and $\psi$. For example, the case $r=\ell-1$ occurs whenever $E[\ell]\subseteq E(k)$, and the case $r=1$ (i.e., $R_\psi$ is a field) occurs when $h(x)$ is irreducible and $f(x)$ is not a square in the field $k[x]/\big(h(x)\big)$. This variability is the reason that we need to develop these tools for \'etale $k$-algebras instead of just using the existing results for fields; determining the decomposition of $R_\psi$ requires factoring $h(x)$, which would dominate the time complexity of the algorithm.

We wish to use $P_\psi$ as the test point to compute $t_\ell$. Lemma \ref{lem:ell_torsion_nice_under_coprime_isogeny} implies that $P_\psi\in E^\circ(R_\psi)$ has all coordinates of prime order $\ell\neq 2$, and this implies that the doubling and isogeny formulas in Proposition \ref{prop:the_usual_algorithms_work} will work whenever $\ell$ is coprime to the degree of the isogeny. However, the condition required by Proposition \ref{prop:the_usual_algorithms_work} to invoke the addition formula may not always be met for points with all coordinates of order $\ell$, and the arithmetic of arbitrary such points of $E^\circ(R_\psi)$ can get quite messy. The following lemma assures us that when it occurs in our computations, this situation resolves as cleanly as possible.

\begin{lem}\label{lem:isogeny_kernels_well_behaved}
    Let $\psi:E \to E/\ker\psi$ an $\ell$-isogeny, and let $P_\psi \in E(R_\psi)$ as above. Let $\phi,\phi':E\to E'$ be separable isogenies 
    that do not factor through $\psi$. Then $\phi(P_\psi),\phi'(P_\psi)\in (E')^\circ(R_\psi)$. 
    If $x(\phi(P_\psi)) - x(\phi'(P_\psi)) \not\in R_\psi^\times$, then either $\phi(P_\psi) = \phi'(P_\psi)$ or $\phi(P_\psi) + \phi'(P_\psi) = 0$.

    It follows from Proposition \ref{prop:the_usual_algorithms_work} that $\phi(P_\psi) + \phi'(P_\psi)=(\phi+\phi')(P_\psi)$ is either in $(E')^\circ(R_\psi)$ and computed via the usual addition laws, or is equal to 0.
\end{lem}
\begin{proof}
    Let $E(R_\psi)\cong \bigoplus_{i=1}^rE(L)$ be the isomorphism induced by our chosen factorization of $R_\psi$, and write $P_\psi = (P_1,\hdots, P_r)$. Our assumption that $\phi$ and $\phi'$ do not factor through $\psi$ implies that no coordinate $P_i$ is contained in the kernels of $\phi$ and of $\phi'$. Thus Lemma~\ref{lem:ell_torsion_nice_under_coprime_isogeny} implies that $\phi(P_\psi)$ and $\phi'(P_\psi)$ are in $E^\circ(R_\psi)$. 
    
    Note that
    \[
        \phi(P_\psi) + \phi'(P_\psi) = \big((\phi+\phi')(P_1), \hdots, (\phi+\phi')(P_r)\big).
    \]
    Suppose that $x(\phi(P_\psi)) - x(\phi'(P_\psi)) \not\in R^\times$. 
    Then for some $i$, we have $x(\phi(P_i))=x(\phi'(P_i))$, so either $\phi(P_i)=\phi'(P_i)$ or $\phi(P_i)=-\phi(P_i)$. In the first case, $P_i\in \ker(\phi- \phi')$, and in the second $P_i \in \ker(\phi+\phi')$; the two cases may be handled identically.

    Base changing to $\bar{k}$, having $P_i\in \ker(\phi- \phi')$ implies that the $[L:k]$ points of $E(\bar{k})$ corresponding to $P_i$ must also lie in $\ker(\phi- \phi')$. Since $\ker \psi \subseteq E(\bar{k})$ is a cyclic subgroup of prime order, this implies that $\ker\psi \subseteq \ker(\phi- \phi')$. Then, since all $\bar{k}$-points corresponding to $P_\psi$ lie in $\ker(\phi-\phi')$, we have $P_\psi\in \ker(\phi-\phi')$ as well, so $\phi(P_\psi) - \phi'(P_\psi) = 0$.
\end{proof}

\begin{rmk}\label{rmk:psi_factors_through_alpha}
    In Schoof's algorithm, the trace of the Frobenius endomorphism $\pi_E$ of $E$ is computed by finding an integer $0\leq t <\ell$ such that $\pi_E^2(P_{[\ell]})+[p] = t\pi_E(P_{[\ell]})$, where $P_{[\ell]}$ is the inclusion $E[\ell]- \{ 0\}\hookrightarrow E$. The coordinate ring of $E[\ell]- \{ 0\}$ is $R_{[\ell]}=k[E]/(\psi_\ell)$ where $\psi_\ell$ is the $\ell$th division polynomial. One attempts to compute the $R_{[\ell]}$-point $\pi_E^2(P_{[\ell]})+[p_\ell](P_{[\ell]})$ using the usual chord-and-tangent law as described above. If this fails at any point, it is because $x(Q)-x(Q')\not\in R_{[\ell]}^{\times}$ for some points $Q,Q'\in E^\circ(R_{[\ell]})$ previously computed. But since $x(Q),x(Q')$ can be represented by polynomials $a(x),a'(x)$ in $x$ with coefficients in $\FF_p$, the polynomial $a(x)-a'(x)$ must have nontrivial gcd $h(x)$ with $\psi_\ell(x)$. This lets us decompose $R_{[\ell]}\simeq R\times R'$ as a product of $k$-algebras and hence $E(R_{[\ell]})\simeq E(R)\times E(R')$. That is, we have identified a rational subvariety of $E[\ell]- \{ 0\}$, cut out by $h(x)$. These are the points $Q$ in the $\ell$-torsion such that $h(x(Q))=0$, corresponding to $R=R_{[\ell]}/(h(x))$.  One can show that $h(x)$ is actually the kernel polynomial of an $\ell$-isogeny, and we are in the setting of Elkies' improvement to Schoof's algorithm. For more details and discussion, we refer the reader to~\cite[Lecture \#8, 8.5]{MITEC}.
\end{rmk}

We now wish to begin establishing the complexity of some elementary operations with $P_\psi$ in $E(R_\psi)$. First, we establish a uniform way of representing the points of $E^\circ(R_\psi)$ which we will care about.

\begin{defn}
    We say that a point $P\in E^\circ(R_\psi)$ is in \textit{standard form} if it is represented by a tuple $(a(x), b(x)y) \in R_\psi^2$ with $\deg a < \deg h$ and $\deg b < \deg h$.
\end{defn}

\begin{lem}
    For every separable $\phi\in \End(E)$ that does not factor through $\psi$, $\phi(P_\psi)$ has a unique standard form. 
\end{lem}
\begin{proof}
    Recall that $P_\psi:\Spec R_\psi \to E^\circ$ corresponds to the quotient map on underlying rings, and thus has coordinates $(x,y)\in R_\psi^2$. Writing $\phi = (\frac{u(x)}{v(x)}, \frac{s(x)}{t(x)}y)$, by Lemma \ref{lem:isogeny_kernels_well_behaved} and Proposition \ref{prop:the_usual_algorithms_work} we know that $\phi(P_\psi) = (u(x)v(x)^{-1}, s(x)t(x)^{-1}y)$.

    First, note that by applying the relations $y^2 = f(x)$ and $h(x) = 0$, any element $q(x,y)\in R_\psi$ may be uniquely reduced to one of the form $q_1(x)y + q_0(x)$ with $\deg q_1$ and $\deg q_0$ less than $\deg h$.
    Let $T$ denote the $k$-subalgebra of $R_\psi$ generated by $x$. Then $T \cap R_\psi^\times$ is closed under taking inverses, since if $r(x)\cdot(q_1(x)y + q_0(x)) = 1$ then $q_1 = 0$ in $R_\psi$. Thus, $u(x)v(x)^{-1}$ and $s(x)t(x)^{-1}$ have unique representatives $a(x)$ and $b(x)$ with the appropriate degrees, and we're done.
\end{proof}

We may now state and prove our desired complexity results.

\begin{prop}\label{prop:P_psi_addition_complexities}
Let $E$ be a supersingular elliptic curves over $
\FF_{p^2}$, and let $\psi\colon E\to E/\ker\psi$ be a separable $\ell$-isogeny. Let $n=\log p$ and assume $\log \ell \in O(n)$. Let $\phi,\phi'\in \End(E)$ be isogenies that do not factor through $\psi$.
\begin{enumerate}[(a)]
\item   Given $\phi(P_\psi)$ and $\phi'(P_\psi)$ in standard form, the standard form of $(\phi+\phi')(P_\psi)$ can be computed with \(O(\ell n(\log\ell n)^2)\) bit operations. When $\ell\in O(n)$, this simplifies to \(O(n^2(\log n)^2)\).
\item Given $\phi(P_\psi)$ and $0\leq c < \ell$, the standard form of \(c\phi(P_\psi)\) can be computed with \(O(\ell n(\log\ell n)^2)\) bit operations. When $\ell\in O(n)$, this simplifies to $O(n^2(\log n)^2)$. 
\end{enumerate}
\end{prop}
\begin{proof}
 Let $h(x)$ be the kernel polynomial of $\psi$. Since $\phi(P_\psi)=(a,by)$ and $\phi'(P_\psi)=(a',b'y)$ with $a,b,a',b'\in k[x]/(h(x))$ are in standard form, we can first test if $\phi(P_\psi)=\pm\phi'(P_\psi)$ by testing if $a=a'$ and $b=-b'$ in $k[x]/(h(x))$. If $\phi(P_\psi)=-\phi'(P_\psi)$, we return $0$, and if $\phi(P_\psi)=\phi'(P_\psi)$ we can use the doubling formulas, by Lemma~\ref{lem:isogeny_kernels_well_behaved}. Otherwise,  $x(\phi(P_\psi)) - x(\phi'(P_\psi)) \in R_\psi^\times$ and we compute its inverse; this takes $O(\ell n(\log\ell n)^2)$ bit operations by Proposition \ref{prop:basic_complexity_table}.  We  then compute $\phi(P_\psi) + \phi'(P_\psi)$ in $O(1)$ multiplications and additions in $R_\psi$, all of which are dominated by the cost of inversion.

To compute $c\phi(P_\psi)$, we can employ the usual double-and-add approach, modified to require only a single inversion. We do this by recording intermediate points $c'P_\psi$ in the form $(\frac{a_1(x)}{a_2(x)}, \frac{b_1(x)}{b_2(x)}y)$ with each $a_i$ and $b_i$ a polynomial in $x$ of degree less than or equal to $\deg h$. As we double these points and add them together, additions and multiplications in $R_\psi$ are used to clear compound fractions and simplify sums of fractions. At the end, two inversions are performed to put $c\phi(P_\psi)$ in standard form. The process up to these last inversions occurs in $O(\log c) = O(\log \ell)$ steps each taking $O(1)$ additions and multiplications in $R_\psi$, and is thus once again dominated by the cost of inversion.
\end{proof}

Finally, we will need to evaluate the endomorphisms $\alpha$ and $\alpha^2$ at the point $P_\psi$. The usual presentation for $\alpha$ will be as a chain of bounded degree isogenies of length $O(\log p)$.  Assuming $\ker\alpha \cap \ker \psi = 0$, Lemma \ref{lem:ell_torsion_nice_under_coprime_isogeny} implies
 the images of $P_\psi$ along this chain remain confined to the usual affine open sets of the appropriate elliptic curve. We analyze the complexity of this operation as follows.

\begin{prop}\label{prop:chain_composition_complexity}
    Let $E_1,\hdots, E_{L+1}$ be supersingular elliptic curves over $\FF_{p^2}$ with $E = E_1 = E_{L+1}$. Let $\alpha\in \End(E)$ be given by a chain of isogenies $\phi_i:E_i\to E_{i+1}$ for $i=1,\ldots, L$ given in standard form.

    Let $d$ be an upper bound on the degree of all $\phi_i$, let $n = \log p$, and assume $\log \ell \in O(n)$. Let $\phi=\phi_L\circ\cdots \phi_1$. If $\ker\phi \cap \ker \psi = \{0\}$, the standard form of $(\phi_L\circ \cdots \circ \phi_1)(P_\psi) \in E^\circ(R_\psi)$ can be computed in
    \[
    O(dL\ell n\log\ell+\ell n(\log\ell)^2)
    \]
    bit operations. In the case that $d \in O(1)$ and $L,\ell \in O(n)$, the total complexity is
    \(
    O(n^3\log n).
    \)
\end{prop}

\begin{proof}
    As in the proof of Proposition \ref{prop:P_psi_addition_complexities}, the main trick here is to save all inversions until the end, as this otherwise dominates the process. Suppose that for some $i$ we have already computed $(\phi_i\circ\hdots\circ \phi_1)(P_\psi)$ in the form $(\frac{a_1(x)}{a_2(x)}, \frac{b_1(x)}{b_2(x)}y)$. Then writing $\phi_{i+1} = (\frac{u(x)}{v(x)}, \frac{s(x)}{t(x)}y)$, we have
    \[
        (\phi_{i+1}\circ\hdots\circ \phi_1)(P_\psi) = \left( \frac{u(\frac{a_1(x)}{a_2(x)})}{v(\frac{a_1(x)}{a_2(x)})},\ \frac{s\big(\frac{a_1(x)}{a_2(x)}\big)}{t\big(\frac{a_1(x)}{a_2(x)}\big)}\frac{b_1(x)}{b_2(x)}y \right).
    \]
    We first clear appropriate powers of $a_2(x)$ to eliminate the compound fractions. Evaluating each numerator and denominator requires computing the first $O(d)$ powers of $a_1(x)$ and $a_2(x)$, then performing $O(d)$ multiplications and additions in $R_\psi$ to simplify. This takes $O(d\ell n\log\ell)$ bit operations by Proposition \ref{prop:basic_complexity_table}, so evaluating the entire chain in this fashion takes $O(dL\ell n\log\ell)$ bit operations. Finally, the two inversions at the end are performed in time $O(\ell n(\log \ell)^2)$, giving the stated complexity bounds.
\end{proof}

The method outlined in Proposition~\ref{prop:chain_composition_complexity} gives an asymptotic time savings over computing the standard forms of the intermediate compositions when $d\in O(1)$ and $L\in O(n)$. Essentially, we are working with projective coordinates, allowing us to compute the result of the chain of isogenies with just $2$ inversions in $R_\psi$ instead of $2L$ inversions in $R_\psi$. If we do not know if $\ker\phi\cap\ker\psi = \{0\}$, then we can still proceed as suggested in the proof; if at some point we attempt to invert a zero-divisor, it means $\ker\psi \subseteq \ker\phi$ and we determine $\phi(P_\psi)=0$.

Our algorithm for computing the trace modulo $\ell$ requires the kernel polynomial of an $\ell$-isogeny of~$E$, just as in the SEA algorithm. The idea is to first compute the modular polynomial $\Phi_{\ell}(X,Y)$, find a root $j'$ of $\Phi_{\ell}(j(E),Y)$ to get the $j$-invariant of a curve $E'$ which is $\ell$-isogenous to~$E$, and then computing the corresponding $\ell$-isogeny using Elkies' algorithm~\cite{Elk98}.
Elkies' algorithm uses the modular polynomial $\Phi_{\ell}(X,Y)$; given a nonsingular zero  $(j,j')$, the algorithm first computes models~$E$ and $E'$ with $j$-invariants $j$ and $j'$ and then computes the kernel polynomial. The algorithm fails when $(j,j')$ is a singular point (but see~\cite{mahaney2024computingisogeniessingularpoints} for an extension in this case). For our purposes, we will just prove that if $p$ is sufficiently large relative to $\ell$ and if $j=j(E)$ is not $0$ or $1728$, then there is at most one singular point of the form $(j(E),j')$ of multiplicity at most $2$. 

\begin{lem}\label{prop:multiple_edges_bound}
    Let~$E$ be a supersingular elliptic curve defined over~$\FF_{p^2}$ and assume $j(E)\notin\{0,1728\}$. Let $\ell<(p/4)^{1/4}$ be a prime. Then there is at most one  root $\tilde{j}$ of $\Phi_{\ell}(j,Y)$ such that the multiplicity of $\tilde{j}$ is larger than $1$, and its multiplicity is at most $2$.  
\end{lem}
    
    \begin{proof}

        Assume $\phi_{\ell}(j(E),Y)$ has a 
        multiple root $j$ of multiplicity 
        $m$. Then there are $m$ cyclic $
        \ell$-isogenies $\phi_i\colon E\to 
        E'$ with $j(E')=j'$, $1\leq i \leq 
        m$, and with distinct kernels. Any 
        pair of distinct $\ell$-isogenies $
        \phi_i\neq \phi_j$ gives rise to a $
        \ell^2$-endomorphism $
        \alpha_{ij}\coloneqq\widehat{\phi_j}\phi_i$ 
        of~$E$. Since we assume $\phi_i$ and $\phi_j$ have distinct kernels and that $j(E)\notin\{0,1728\}$, we have that $\phi_j\neq u\phi_iv$ for any automorphisms $u\in \Aut(E')$ and $v\in \Aut(E)=\{[\pm1]\}$. Thus $\alpha_{ij}$ is a 
        cyclic $\ell^2$-endomorphism of~$E$. 
        
        Suppose toward contradiction that $\alpha_1,\alpha_2$ are two of the $\alpha_{ij}$ and that they do not commute. We will argue as in the proof of of~\cite[Theorem~2']{kaneko} that this implies $p\leq 4\ell^2$. Let $D_1,D_2$ be the discriminants of $\alpha_1,\alpha_2$. In general, if $x$ and $y$ are elements  of a quaternion algebra such that $x$ is invertible and $xy$ is central, then $x$ and $y$ commute. Thus $\beta=(\alpha_1-D_1/2)(\alpha_2-D_2/2)$ is not central and therefore has negative discriminant. A calculation shows that the discriminant of $\beta$  is $D=((2\tr(\alpha_1\alpha_2)-D_1D_2)^2-D_1D_2)/4<0$, and another calculation shows that the discriminant of the quaternion order $\OO$ generated by $\alpha_1$ and $\alpha_2$ is $(-D/4)^2$. Since $\End^0(E)$ has discriminant $p$, we have that $p^2|(-D/4)^2$ since $\OO$ is contained in some maximal order whose discriminant must be $p^2$. Now observe that $-D\leq D_1D_2$ so we conclude \[p\leq \frac{-D}{4} \leq \frac{D_1D_2}{4}\leq 4\ell^4\] since $D_1,D_2\leq 4\ell^2$, contradicting our assumption that $\ell<(p/4)^{1/4}$. Therefore any two $
        \alpha_{ij}$ must commute and therefore are integral elements of norm $\ell^2$ all contained in the same imaginary quadratic field. We now must count the maximal 
        number of elements of norm $\ell^2$ in the quadratic imaginary order $A=\ZZ[\alpha]$  contained in 
        $\End(E)$, where $\alpha=\alpha_{12}$. Let $K=\QQ(\alpha)$ and let $
        \OO_K$ be the ring of integers of 
        $K$.

        The order $\ZZ[\alpha]$ is maximal at $\ell$, i.e., the conductor of $\ZZ[\alpha]$ is coprime to $\ell$, because $\disc\alpha$ is coprime to $\ell$. Indeed, if $\ell\mid\disc\alpha$ then $\ell\mid\tr\alpha$ since $\deg\alpha = \ell^2$. This implies $\alpha^2=\ell(\alpha + \ell)$, so $\alpha^2(E[\ell])=0$. This would imply $\phi_2\widehat{\phi_1}=u\ell$ for some $u\in\Aut(E')$, which would imply 
        \[
        \phi_1\widehat{\phi_2} = \ell u = \phi_1\circ\widehat{\phi_1}\circ u.\]
        This would imply 
        $\phi_2 = u^{-1}\circ \phi_1$, contradicting the assumption that $\phi_1$ and $\phi_2$ have distinct kernels. 
        Let $A=\ZZ[\alpha]$.

        We claim that the prime $\ell$ splits in $A$. Since the conductor of $A$ is coprime to $\ell$, the ideal $\ell A$ factors uniquely into a product of primes of $A$ coprime to the conductor of $A$, and the primes $\pp$ appearing in this factorization are in bijection with the primes $\pp A$ appearing in the factorization of $\ell \OO_K$.   Let $\mathfrak{l}$ be a prime above $\ell$ in $A$.  If  $\ell$ ramifies in $\OO_K$ then since $\nrd\alpha = \ell^2$ we have 
        \[
        \alpha\OO_K = \mathfrak{l}^2\OO_K=\ell\OO_K
        \]
        implying 
        \[
        \alpha A=\mathfrak{l}^2=\ell A
        \]  which shows $\alpha=u\ell$ for some $u\in A^{\times}\subseteq\Aut(E)=\{[\pm1]\}$, contradicting the fact that $\ker\alpha$ is cyclic. Similarly  $\ell$ is not inert, i.e., $\mathfrak{l}\neq \ell A$, because again this would imply $\alpha A=\ell^2 A$ which would imply $\alpha$ does not have cyclic kernel. So $\ell$ must split in $A $ as $\ell A = \mathfrak{l}\overline{\mathfrak{l}}$, and  $\alpha A=\mathfrak{l}^2$. 
    
        The only ideals of norm $\ell^2$ are $\ell A, \mathfrak{l}^2,\overline{\mathfrak{l}}^2$, hence the only elements of norm $\ell^2$ are $\pm\ell, \pm\alpha, \pm\widehat{\alpha}$. Thus, up to post-composition with an automorphism, there exist at most two cyclic $\ell$-isogenies of~$E$ that extend to cyclic endomorphisms of~$E$ of degree $\ell^2$. 
    
    \end{proof}
Our earlier lemmas on the complexity of computing $\alpha(P_\psi)$ for an $\ell$-isogeny $\psi$ require the assumption that $\ker\psi \cap \ker \alpha = 0$. As mentioned earlier, if we can find such an $\ell$-isogeny, we can compute $t_\ell$ by comparing $(\alpha^2+[\deg\alpha])(P_\psi)$ against multiples of $\alpha(P_\psi)$. However, if the computation of the standard form of $\alpha(P_\psi)$ fails at some point, it implies $\ker\psi\subseteq \ker\alpha$ and then we can't determine $t_\ell$ using $\psi$. The proposition below implies we need to find at most $2$ $\ell$-isogenies of $E$ to determine $t_\ell$. 
\begin{prop}\label{prop:char_poly_identity}
    Let $p>3$ be prime, let~$E$ be an elliptic curve over $\FF_{q}$, let $\ell\not=p$ be an odd prime, let $\psi\colon E\to E/\ker\psi$ be an $\ell$-isogeny. If $\ker\alpha\cap \ker\psi=0$,  the  integer $0\leq t \leq \ell-1$ satisfies 
    \[
    \alpha^2(P_\psi) + [\deg \alpha](P_\psi) = t\alpha(P_\psi)
    \]
    if and only if $\tr \alpha \equiv t \pmod{\ell}$. If $\psi_1$ and $\psi_2$ are two  $\ell$-isogenies of~$E$ with distinct kernels, both contained in $\ker\alpha$, then $\tr\alpha \equiv 0 \pmod{\ell}$. 
\end{prop}

\begin{proof}
        First suppose $\alpha(\ker\phi)$ is nonzero. Then  $\alpha^2(P_\psi)+\deg(\alpha) P_\psi = t\alpha(P_\psi)$ holds as an equality in $E(R_\psi)$ if and only if $t\equiv \tr\alpha\pmod{\ell}$ since $\alpha^2(P_\psi)+\deg(\alpha )P_\psi = (\tr\alpha) \alpha(P_\psi)$ and $\alpha(P_\psi)$ has order $\ell$ in $E(R_\psi)$.
If $\alpha(P_{\psi_i})$ is zero for two isogenies $\psi_1,\psi_2$ with distinct kernels, then $\alpha(E[\ell])=0$ because the $\ell$-torsion is generated by the corresponding kernels and $\alpha$ is zero on the generators. This implies there exists $\beta\in \End(E)$ such that $\alpha=\ell\beta$, so 
    \[
        \tr\alpha = \tr(\ell\beta) = \ell\tr\beta \equiv 0 \pmod{\ell}.
    \]
\end{proof}

We are now ready to prove the main theorem of this section on the complexity of computing the trace reduced modulo $\ell$ of a supersingular endomorphism. 
\begin{thm}\label{thm:TraceModEll}
    Let $\alpha=\phi_L\circ\cdots\circ\phi_1$ be an endomorphism of a supersingular elliptic curve~$E$ defined over~$\FF_{q}$, let $n=\ceil{\log p}$, and let $\ell\in O(n)$ be an odd prime  Let $d=\max\{\deg\phi_i\}$.  Then $t_{\ell}\coloneqq \tr\alpha\bmod{\ell}$ can be computed in an expected $O(n^3(\log n)^3 + dLn^2\log n)$ bit operations. 
    Assuming GRH, the space complexity is $O(\ell^3\log\ell 
    + dLn 
    + \ell n    
    )$. The time and (conditional) space complexities simplify to $O(n^3(\log n)^3)$ and  $O(n^3\log n)$, respectively, when $d\in O(1)$ and $L,\ell\in O(n)$. 
\end{thm}

\begin{proof}
    All isogenies of~$E$ will be defined over an extension of degree either $1,2$, or $3$ of $\FF_q$, and the cost of arithmetic operations in this extension increase only by a constant factor. We will proceed with assuming~$E$ already has all its isogenies defined over~$\FF_q$.

    To compute a kernel polynomial for an $\ell$-isogeny $\psi$ of~$E$, we first compute $\Phi_{\ell}(X,Y)$ a cost of $O(\ell^3(\log \ell)^3)$ bit operations; this can be done using the algorithm of Kunzweiler and Robert~\cite[Theorem~5.2, Remark~5.3]{KR24}. With $\Phi_{\ell}$ we can compute $\varphi_{\ell}(Y)\coloneqq \Phi_{\ell}(j(E),Y)$ by evaluating at $X=j(E)$ in time $O(\ell^2M(\ell \log \ell + n))\subseteq O(n^3\log n)$. Alternatively, one could first compute the instantiated modular polynomial $\Phi(Y)\coloneqq \Phi_{\ell}(j(E),Y)$, along with the partial derivatives $\Phi_X(Y)$ and $\Phi_{XX}(Y)$, in $O((\log q)^3(\llog q)^3)$ expected time and $O(\ell^3\log\ell)\subseteq O(n^3\log n)$ space under GRH~\cite[Algorithm~1, Theorem~4]{SutherlandEvaluation}. Next, compute a kernel polynomial $h_1$. To do this,  compute a random root $j_1\in \FF_{p^2}$ of $\Phi_{\ell}(j(E),Y)$ in expected $O(\ell n^2\log(\ell n)) \subseteq O(n^3\log n)$ time. The root $j_1$ that is necessarily the $j$-invariant of a supersingular elliptic curve $E_1$ with the property that there exists an $\ell$-isogeny $\psi\colon E\to E_1$, and this isogeny is defined over~$\FF_{p^2}$ since~$E$ is supersingular. Assuming for the moment $j(E_1)$ is a simple (i.e., multiplicity $1$) root of $\Phi(Y)$, we proceed to compute $h_1$ by first computing a model for the curve $E_1$ such that the isogeny $\psi_1\colon E\to E_1$ is {\em normalized}, meaning $\phi^*\omega_1=\omega$ where $\omega,\omega_1$ are the invariant differentials $dx/2y$ of~$E$ and $E_1$ respectively.  This is done via Elkies' algorithm~\cite{Elk98}; see also~\cite[Section~7]{SchoofSEA} and~\cite[Chapter~25, Algorithm~28]{GalPKC} for details. The model $E_1$ can be computed  with $O(\ell)$ operations in $\FF_{p^2}$. The isogeny $\psi\colon E\to E'$ can then be computed in $O(\M(\ell)\log \ell)$ operations in $\FF_{p^2}$ with the algorithm of Bostan, Morain, Salvy and Schost~\cite[Theorem~1]{BMSS}. If $j_1$ is not a simple root of $\Phi(Y)$, we compute random roots of $\Phi(Y)/(Y-j_1)^e$, where $e$ is the multiplicity of $j_1$ as a root of $\Phi(Y)$, until finding a simple root. Assuming $\ell \leq p/4$, there are at most $2$ roots of $\Phi(Y)$ which are not simple roots, counted with multiplicity, by Proposition~\ref{prop:multiple_edges_bound}.

    We then compute $\alpha(P_\psi)$, requiring $O(dL\ell n\log\ell + \ell n (\log \ell)^2)$ bit operations. If the computation fails, it means $\ker\psi \subseteq \ker \alpha$, so 
    we repeat the above, computing a new $\ell$-isogeny $\psi'$ with $\ker\psi'\not=\ker\psi$ and $\alpha(P_{\psi'})$. If $\alpha(P_{\psi'}=0$ a second time, then $\tr\alpha\equiv 0 \pmod{\ell}$ by Proposition~\ref{prop:char_poly_identity}. So suppose, relabeling if necessary, that $\alpha(P_\psi)\in E^\circ(R_\psi)$. So suppose $\alpha_{h}\neq 0$ for $h=h_1$ or $h_2$. We can then compute $(\alpha^2)(P_\psi)$ and $[\deg\alpha]P_\psi$. The cost is dominated by the $O(de\ell n\log \ell +\ell n (\log \ell)^2)$ bit operations required to compute $\alpha_h$. $\alpha^2(P_\psi)=-[\deg\alpha](P_\psi)$ we return $0$. Otherwise we compute  $(\alpha^2+[\deg\alpha])(P_\psi)$ and $c\alpha(P_\psi)$ for increasing $c=1,2,\ldots,\ell-1$ until $c\alpha(P_\psi)=(\alpha^2+[\deg\alpha])(P_\psi)$, requiring $O(\ell^2n(\log\ell)^2)$ operations by Proposition~\ref{prop:chain_composition_complexity}. The algorithm is correct  by Proposition~\ref{prop:char_poly_identity}. 
\end{proof}

\subsection{Computing the trace modulo \texorpdfstring{$p$}{p}} \label{subsec:mod_p}
In this section, we give a fast algorithm for computing the trace modulo $p$ of a separable endomorphism of an elliptic curve over a field of characteristic $p$. If $E\cong \CC/\Lambda$ is an elliptic curve over $\CC$, then $\alpha\in \End(E)$ lifts to a map $\CC\to \CC$ given by multiplication   by a complex number $a$. The complex number $a\in \CC$ satisfies  $\deg \alpha = a\overline{a}$ and $\tr\alpha = a + \overline{a}$, where $\overline{\ \cdot \ }$ is complex conjugation. The element $a$ is the scaling factor in $\CC$ such that $\alpha^*\omega = a\omega$, where $\omega$ is any invariant differential on~$E$. Thus $\tr\alpha$ may be computed by computing the trace of the scaling factor of $\alpha$\footnote{The idea to compute $\tr\alpha$ from the action of $\alpha$ on an invariant differential appears to be known to experts; see~\cite{cremonasageticket}.}. We'll adapt this idea to the case that $\alpha$ is a separable endomorphism of a supersingular elliptic curve.

Let~$E$ be a supersingular elliptic curve defined over~$\FF_{p^2}$ and let $\omega$ be an invariant differential on~$E$. By~\cite[Corollary III.5.6]{aec}, the map  
\begin{align*}
    \End(E) &\to \FF_{p^2} \\ 
    \alpha &\mapsto a_{\alpha},
\end{align*}
a ring homomorphism, where $a_{\alpha}\in \FF_{p^2}$ satisfies 
\[
\alpha^*\omega = a_{\alpha} \omega.
\]

\begin{prop}\label{prop:trace_on_LieE}
Let $q$ be a power of a prime $p$  and let~$E$ be an elliptic curve over $\FF_q$. Suppose $\alpha \in \End(E)$.
    Then modulo~$p$ we have
\[ 
\tr\alpha \equiv \begin{cases}
    a_\alpha + a_\alpha^{-1}\deg\alpha & \text{if } \alpha \text{ is separable;} \\
    a_{\widehat{\alpha}} & \text{if } \alpha \text{ is inseparable.}
\end{cases}
\]

Suppose now that~$E$ is supersingular. Then the above simplifies to \[ \tr \alpha \equiv \Tr_{\FF_{p^2}/\FF_{p}}(a_{\alpha})\pmod{p},\] where $\Tr_{\FF_{p^2}/\FF_p} a = a+a^p$ is the field trace of the extension $\FF_{p^2}$ of $\FF_p$. 
\end{prop}
\begin{proof}
As stated above, the map $\alpha\mapsto a_{\alpha}$ is a ring homomorphism. 
     The endomorphism $\alpha$ satisfies its characteristic polynomial $\chi_{\alpha}(x)\coloneqq x^2-(\tr \alpha) x + \deg\alpha$, so 
    \[
    a_{\alpha}^2-(\tr \alpha)a_{\alpha}+\deg\alpha=0
    \]
    holds as an identity in $\FF_{q}$. Similarly,  $a_{\widehat{\alpha}}$ is also a root of $\chi_{\alpha}$ in $\FF_q$. 
    The kernel of the map $\alpha\mapsto a_{\alpha}$ consists precisely of the inseparable endomorphisms of~$E$~\cite[Corollary~III.5.6(b)]{aec}, so $a_\alpha=0$ for all inseparable endomorphisms. If $\alpha$ is inseparable, the polynomial $\chi_{\alpha} \bmod{p}$ reduces to $x(x-\tr\alpha)$ modulo $p$. Since $a_{\alpha}=0$ is one root of $\chi_{\alpha}$ in $\FF_q$, the other root is  $a_{\widehat{\alpha}} = \tr\alpha$. If $\alpha$ is separable, then $a_\alpha \neq 0 \in \FF_q$. Since the map $\alpha\mapsto a_{\alpha}$ is a ring homomorphism, we have 
    \[
    a_{\alpha}a_{\widehat{\alpha}} = a_{\deg\alpha} = \deg\alpha 
    \]
    as an identity in $\FF_q$, from which the congruence in the proposition follows.

    Now assume that~$E$ is supersingular. We will show \[ \tr \alpha \equiv \Tr_{\FF_{p^2}/\FF_{p}}(a_{\alpha})\pmod{p}.\] That $a_{\alpha}\in \FF_{p^2}$ is clear since $a_{\alpha}$ satisfies a quadratic polynomial with coefficients in $\FF_p$. 
    If $\alpha=m\in \ZZ$, then $\widehat{m}=m$ and $[m]^*\omega=m\omega$, as $\omega$ is an invariant differential. Thus 
    \[
    a_{\widehat{m}} = a_{m} = m \equiv m^p \pmod{p}.
    \]
    Now assume  that $\alpha\not \in \ZZ$.
    Then the prime~$p$ cannot split in $\ZZ[\alpha]$, so the polynomial~$\chi_{\alpha}$ has either zero or one roots in $\FF_p$. Suppose first that $\chi_{\alpha}$ has no roots in $\FF_p$, i.e., $p$ is inert in $\ZZ[\alpha]$. Then $\chi_{\alpha} \bmod{p}$ is the minimal polynomial of $a_{\alpha}$, so
    \[
    \tr \alpha \equiv \Tr_{\FF_{p^2}/\FF_{p}} a_{\alpha} \pmod{p}
    \]
    as desired. If instead $\chi_{\alpha}$ has exactly one root $a_{\alpha}=a_{\widehat{\alpha}}=m\in \FF_p$, 
    we conclude 
    \[
    \tr \alpha \equiv a_{\tr \alpha} \equiv a_{\alpha}+a_{\widehat{\alpha}} \equiv 2m \equiv m+m^p \equiv \Tr_{\FF_{p^2}/\FF_{p}} a_{\alpha}\pmod{p}.
    \qedhere
    \]
\end{proof}

\begin{rmk}
    It is not true that the reduced trace of $\alpha$ is congruent modulo $p$ to the field trace of $a_{\alpha}$ when $\alpha$ is an endomorphism of an ordinary elliptic curve~$E$. For example,  let $\alpha=\pi_E$ be the Frobenius endomorphism of~$E$. Then $a_{\pi_E}=0$ and thus has trace zero but $\tr  \alpha\not\equiv 0\pmod{p}$ since~$E$ is ordinary. 
\end{rmk}

\begin{rmk}
When~$E$ is supersingular, Proposition~\ref{prop:trace_on_LieE} shows the map $\alpha\mapsto a_{\alpha}$ is a homomorphism of algebras with involution, where the involution on $\End(E)$ is the dual map and the involution on $\FF_{p^2}$ is the $p$-power Frobenius automorphism, i.e., the generator of $\Gal(\FF_{p^2}/\FF_p)$. More precisely, we claim that $a_{\widehat{\alpha}}=a_{\alpha}^p$. This can be seen directly via the Deuring Lifting Theorem~\cite[Theorem~14]{ellipticfunctions}: given the endomorphism $\alpha\in \End(E)$, there is a number field $L$, an elliptic curve $A$ over $L$ with endomorphism $\phi$ and a prime $\Pp$ of $L$ above $p$ such that $A$ and $\phi$ reduce modulo $\Pp$ to~$E$ and $\alpha$, respectively. There is an isomorphism $\End(A)\to \OO$ where $\OO$ is an imaginary quadratic order given by $\theta\mapsto a_{\theta}$.  This isomorphism respects the involutions on $\End(A)$ and $\OO$: we have $a_{\widehat{\theta}}= \overline{a_{\theta}}$. Let $\pp = \Pp\cap \OO$. Suppose the prime $p$ is inert in $\OO$, so the Frobenius element at $p$ corresponds to complex conjugation. This implies 
\[
a_{\widehat{\theta}}=\overline{a_{\theta}} \equiv a_{\theta}^p \pmod{\pp}.
\]
We also have $a_{\phi}\equiv a_{\alpha}\pmod{\pp}$ since $\phi$ reduces to $\alpha$ modulo $\Pp$. Together, the congruences imply $a_{\widehat{\alpha}} \equiv a_{\alpha}^p\pmod{p}$ as claimed. If $p$ is ramified in $\OO$, the same conclusion holds, since \[a_{\widehat{\theta}}=\overline{a_{\theta}}\equiv a_{\theta} \equiv a_{\theta}^p\pmod{\pp}.
\]When $\alpha$ is inseparable, so too is $\widehat{\alpha}$ since~$E$ is supersingular. Then $a_{\alpha}=a_{\widehat{\alpha}}=0$, so the claim holds in this case as well. 
\end{rmk}

Proposition~\ref{prop:trace_on_LieE} suggests a simple approach to computing the trace of a separable endomorphism $\alpha$ modulo $p$:  compute the action of $\alpha$ on the space of invariant differentials. In particular,  compute $a_{\alpha}\in \FF_{p^2}$ and then compute $\Tr_{\FF_{p^2}/\FF_p}a_{\alpha}= a_{\alpha}+a_{\alpha}^p$. Recall that we assume $\alpha$ is represented by a sequence of isogenies, which in turn are represented by rational maps. If 
$\alpha = \phi_L\circ \cdots\circ \phi_1$ where $\phi_i\colon E_i\to E_{i+1}$ are isogenies, and if we choose the standard invariant differential $\omega_i=dx/2y$ on each~$E_i$, then~$\phi_i^*\omega_{i+1}=c_i\omega_{i}$ for some $c_i\in \FF_{p^2}$ for $i=1,\ldots,L-1$. Then 
\[
a_{\alpha} = \prod_{i=1}^{L} c_i.
\]
Thus to compute $a_{\alpha}$ we must compute $c_i$.

\begin{lem}\label{lem:normalizing_constant_from_numerator}
    Let~$E$ and $E'$ be elliptic curves in short Weierstrass form and let~$E$ be given by $y^2=x^3+Ax+B$. 
    Let $\phi\colon E\to E'$ be an $\ell$-isogeny in standard form  $\phi(x,y)=(u/v,cys/t)$. Then the leading coefficients of $u(x)$ and  $s(x)$ are both equal to $c^{2}$. In particular, if $c_x$ and $c_y$ are the leading coefficients of $u(x)$ and $cs(x)$, respectively, then $c_x=c^2$ and $c_y=c^3$. In particular $c=c_x^{-1}c_y$.  
\end{lem}

\begin{proof}
Let $z=-x/y$ be a uniformizer at $0_E$. In the formal group of $E$ at $0_E$ we have 
    \[
    x = \frac{1}{z^2} + O(z^2), \quad y = \frac{-x}{z} = \frac{-1}{z^3}+O(z), \quad \frac{dx}{2y} = (1+O(z))dz.
    \]
    Also let $X,Y$ be coordinates on $E'$ so that $\phi^*X=u(x)/v(x)$ and $\phi^*(Y)=cys(x)/t(x)$, and let $Z=-X/Y$ be a uniformizer at $0_{E'}$, so $\omega_{E'}=(1+O(Z))dZ$. 
    Since $\phi$ is in standard form, we have that $s/t=(u/v)'$, so $\phi^*\omega_{E'}=c^{-1}\omega_E$ and 
    $\phi^*Z=c^{-1}z+O(z^2)$. We first calculate $\phi^*X$ using the formula for $\phi$:
    \[
    \phi^*X=\frac{u}{v}= \frac{c_x}{z^2}+\ldots.
    \]
    Now using $X=Z^{-2}+O(Z^2)$ we have 
    \[
    \phi^*X = \frac{1}{\phi^*Z^2} + O(\phi^*Z^2) = \frac{c^2}{z^2}+O(z^2),
    \]
    so $c_x=c^2$. Similarly, using $Y=-Z^2+O(Z)$ and $\phi^*Y=cys(x)/t(x)$, we have 
    \[
    \frac{-c_y}{z^3}+\ldots = \phi^*Y = \frac{-1}{\phi^*Z^3} +\ldots = \frac{-c^3}{z^3}+\ldots,
    \]
    so $c_y=c^3$. 
\end{proof}

Lemma~\ref{lem:normalizing_constant_from_numerator} implies that, given an isogeny $\phi(x,y)=(I_x,cyI_y)$ in standard form, we may calculate $c$ by finding the leading coefficients $c_x$ and $c_y$ of the numerators of $I_x$ and $cI_y$ and then calculating $c=c_x^{-1}c_y$.

\begin{prop}\label{prop:TraceModPalg}
Let $p>3$ be a prime and let~$E$ be a supersingular elliptic curve defined over~$\FF_{q}$. Let $\alpha\in \End(E)$ be an endomorphism. Let $\alpha = \phi_L\circ\cdots\circ\phi_1\circ \pi_{p^r}$ where $\phi_i$ are isogenies of degree at most $d$ defined over~$\FF_{q}$ and $\pi_{p^r}$ is the $p^r$-power Frobenius. Let $(\phi_x(x),c_iy\phi_y(x))$ be a representation of $\phi$ in standard form. Let $n=\ceil{\log q}$. Then $\tr\alpha\bmod{p}$ can be computed with $O(Ldn + Ln\log n + n(\log n)^2)$ bit operations.

\end{prop}

\begin{proof}
    If $\alpha$ is inseparable, i.e., $r\geq 1$, then $\tr\alpha\equiv 0 \pmod{p}$ by Proposition~\ref{prop:trace_on_LieE}. Suppose $r=0$ so $\alpha$ is separable. 
    For $i=1,2,\ldots,e$, let $c_{i,x}$ and $c_{i,y}$ be the leading coefficients of the numerators of $\phi_{i,x}$ and $c\phi_{i,y}$ respectively. By Lemma~\ref{lem:normalizing_constant_from_numerator}, we have $c_i=c_{i,x}^{-1}c_{i,y}.$ We then have $a_{\alpha}=\prod_i c_i=\big(\prod_i c_{i,x}\big)^{-1}\big(\prod_i c_{i,y}\big)$. Extracting the coefficients $c_{i,x}$ and $c_{i,y}$ from $\phi_{i,x}$ and $\phi_{i,y}$ can be done in time $O(Ldn)$, i.e., the time required to read in the input. Finally $a_{\alpha}$ may be computed with $O(L)$ multiplications and one inversion in $\FF_{p^2}$ at a cost of $O(Ldn + L\M(n)+\M(n)\log n)\subseteq O(Ln\log n + n(\log n)^2)$ bit operations.
    This correctly computes $\tr\alpha\bmod{p}$ by  Proposition~\ref{prop:trace_on_LieE}. 
\end{proof}
\begin{rmk}
    In practice, an endomorphism $\alpha$ will be represented by a sequence  of $L$ isogenies of small (likely prime) degree; presumably these isogenies are computed by V\'{e}lu's or Kohel's formulas and therefore are normalized. Thus one can expect the first  $L-1$ will be normalized and $\phi_L$  is not. Thus we can find $a_{\alpha}$ by calculating the normalizing coefficient of the final isogeny $\phi_L$. 
\end{rmk}

\subsection{Computing with rational points} \label{subsec:mod_points}

We can exploit the fact that~$E$ is supersingular in another way. For simplicity, assume~$E$ is defined over~$\FF_{p^2}$ and that $j(E)\notin\{0,1728\}$. Then $\pi_E=[p]$ or $[-p]$ so $\#E(\FF_{p^2}))=(p-1)^2$ or $(p+1)^2$ and  as abelian groups we have $E(\FF_{p^2}) \cong \ZZ/ (p-1) \times \ZZ/ \ZZ(p-1)$ or $\ZZ/(p+1)\ZZ\times \ZZ/(p+1)\ZZ$ accordingly~\cite[Lemma~4.8]{Schoofnonsingularplanecubics}. We can easily decide which is the case given~$E$, for instance by testing whether a random point has order dividing $p+1$. Assume $\#E(\FF_{p^2})=(p+1)^2$. 

Suppose $\ell^e\mid(p+1)$. Then we can compute the trace $t_{\ell}$ as follows:
\begin{enumerate}
    \item Find a point $P$ of order $\ell^e$;
    \item compute $Q = \alpha(P)$ and $R = (\alpha^2+\deg(\alpha))(P) $; 
    \item Find $d = \operatorname{dlog}_{Q} (R)$;
    \item Then $d \equiv \tr \alpha \pmod{\ell^{e'}}$, where $\ell^{e'}$ is the order of $\alpha(P)$. (Note $e' \leq e$.)
\end{enumerate}
Even better, one can compute $t \bmod \ell^e$ using the discrete logarithm between $Q = P$ and ${R = \alpha(P) + \widehat{\alpha}(P)}$. In a group of order $\ell^k$, discrete logarithms can be computed with $O(k\log\ell \log\log(\ell^k))$ group operations using the Pohlig--Hellman algorithm as described in~\cite[\S\,11.2.3]{Shoupcomputationalintro}. Since $\ell^k\in O(p)$ this requires $O(n\log n)$ group operations totaling $O(n^2(\log n)^2)$ bit operations, where $n=\ceil{\log p}$. Therefore, this procedure is faster than the algorithm in~\ref{subsec:mod_ell} for primes $\ell$ dividing $p-1$. Moreover, we recover more information than before: we can recover $\tr \alpha$ modulo the order of $P$.

When $\ell\mid(p-1)$,  the quadratic twist of~$E$ has a point of order $\ell^k$ for some $k\geq 1$. Conjugating $\alpha$ by a twisting isomorphism preserves its trace, so we can again compute the trace of $\alpha$ modulo a power of $\ell$ efficiently by solving a discrete logarithm between points of order $\ell^k$ on the twist. In isogeny-based cryptography, one often uses primes such that $p^2 -1$ has a large smooth factor. Therefore, this trick will be particularly useful for cases of practical interest.

\section{Computing the trace}\label{sec:trace}
We now describe our algorithm for computing the trace of an endomorphism $\alpha\in \End(E)$ of a supersingular elliptic curve~$E$ defined over~$\FF_q$ where $q$ is a power of a prime $p>3$. As in the SEA algorithm, we compute $\tr\alpha$ by computing $t_{\ell}\coloneqq \tr\alpha\bmod{\ell}$ for enough primes $\ell$ in order to recover $\tr\alpha$ with the Chinese Remainder Theorem.

\begin{thm}\label{thm:cycletrace}
    Let $\alpha=\phi_L\circ\cdots\circ\phi_1$ be a separable endomorphism of a supersingular elliptic curve~$E$ defined over~$\FF_{q}$ with $j(E)\notin\{0,1728\}$. Let $n=\ceil{\log q}$. Assume that $L\log d \in O(n)$. Then $\tr\alpha$ can be computed with 
    $O(n^4(\log n)^2 + dLn^3)$ bit operations. When $d\in O(1)$ and $L\in O(n)$, the complexity is $O(n^4(\log n)^2)$. 
\end{thm}

\begin{proof}
Let $t_{\ell}\coloneqq \tr\alpha\bmod{\ell}$. 
We need $t_{\ell}$ for $\ell<B$ such that $\prod_{\ell\leq B} \ell >4d^{L/2}>4\sqrt{\deg\alpha}$. By the Prime Number Theorem, we may take $B\in O(L\log d)\subseteq O(n)$, so the largest prime used is $O(n)$. 
    We can compute $t_{\ell}=\tr \alpha \bmod{\ell}$ in time 
    $O(n^3(\log n)^3 + den^2\log n)$ by Theorem~\ref{thm:TraceModEll}. The number of primes used is $\pi(B)\in O(n/\log n)$. The total cost is $O(n^3(\log n)^2+dLn^3)$. 
\end{proof}
\begin{rmk}
    If $\alpha\in \End(E)$ is not separable, then $\alpha$ factors as $\alpha= \alpha_s\circ \pi_{p^r}=\phi_L\circ\cdots\circ\phi_1\circ \pi_{p^r}$. We can compute $t_{\ell}=\tr\alpha\bmod{\ell}$ by first computing   $\pi_{p^r}(P_\psi)$ for some  $\ell$-isogeny $\psi$ as in the SEA algorithm, i.e., by computing $x^{p^r}$ and $f(x)^{(p^r-1)/2}$ modulo the kernel polynomial $h$ of $\psi$ with the square-and-multiply algorithm, and then computing $\alpha = \alpha_s(\pi_{p^r}(P_\psi))$ as in Section~\ref{subsec:mod_ell}.  This has the same complexity (see e.g.~\cite[Theorem~13]{ShSureductions}) under the assumption that $r\in O(L)$. Thus we can compute the trace of any endomorphism of a supersingular elliptic curve in time $O(n^3(\log n)^2+dLn^3)$ assuming $L\log d \in O(n)$ and $r\in O(L)$. 
\end{rmk}

\section{Timings}\label{sec:timings}

Our implementation of our algorithm in Sage~\cite{sagemath} is available at \url{https://github.com/travismo/beyond-the-SEA/}. To demonstrate the asymptotic speedup from Elkies primes and the practical speedup from computing the trace modulo $p$, we ran the following experiment. For each bit length $b\in \{16,\ldots,32\}$, we computed a random $b$-bit prime $p$, $5$ supersingular elliptic curves $E/\FF_{p^2}$, and for each curve, we computed an endomorphism of degree $2^L$ for $L=4\ceil{\log_2 p}$ using the cycle-finding algorithm in~\cite{EHLMP20}. We then computed the trace of $\alpha$ with Schoof's algorithm (i.e., using division polynomials), with the SEA-algorithm (i.e., using kernel polynomials of $\ell$-isogenies), and then the SEA algorithm with the trace-modulo-$p$ algorithm, and finally the SEA algorithm, trace-modulo-$p$, and using points whenever $\ell\mid\#E(\FF_{p^2})$. We observe a substantial speedup from using kernel polynomials rather than division polynomials. We also see the speedup offered from computing the trace modulo $p$: we very efficiently get about $\log p$ bits of the trace. Since the input endomorphism always has degree $2^{4\ceil{\log p}}$, and we need to compute the trace modulo $N=p\prod_{\ell<B}\ell$ so that $N>4\cdot 2^{2\ceil{\log p}}$ when using the trace-modulo-$p$ algorithm, we might expect that calculating the trace using the modulo $p$-information makes the cost of a $32$-bit instance roughly as expensive as a $16$-bit instance that does not use the modulo-$p$ information and computes the trace modulo $N'=\prod_{\ell<B} \ell$ so that $N'>4\cdot 2^{2\ceil{\log p}}$. The data in the second figure supports comparison. 

\begin{figure}
\begin{minipage}{.48\textwidth}
    \centering
    \includegraphics[width=0.99\linewidth]{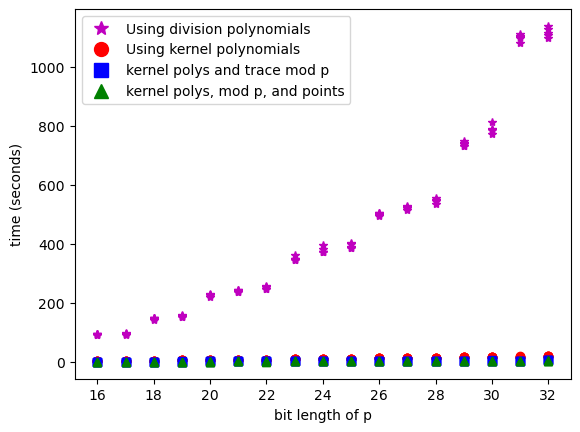}
    \caption{All four methods.}
    \label{fig:all_timings}
    \end{minipage}
    \hfill
    \begin{minipage}{.48\textwidth}
    \centering
    \includegraphics[width=0.99\linewidth]{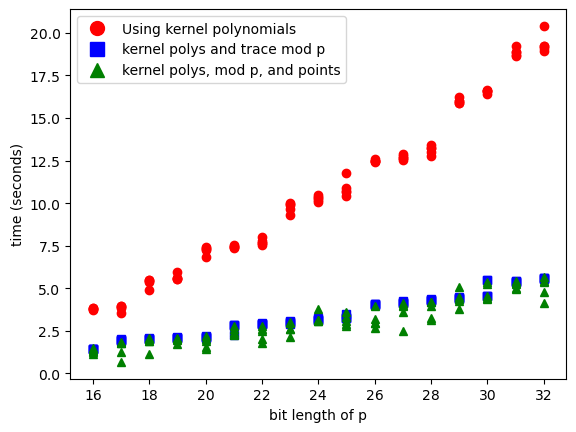}
    \caption{The same timings for methods $2$ through $4$.}
    \label{fig:fast_timings}
    \end{minipage}
\end{figure}

\section{Dewaghe's algorithm for Elkies' method applied to Atkin primes}\label{sec:dewaghe}

We conclude with a discussion connecting the material in Section~\ref{subsec:mod_ell} with the method of Dewaghe for extending Elkies' method for computing $t_{\ell} = \tr\pi_E\bmod{\ell}$ for an ordinary curve~$E$ and Atkin prime $\ell$ for~$E$.  For simplicity, suppose~$E$ is defined over~$\FF_p$. Additionally assume  $j(E)\notin\{0,1728\}$. Define $\varphi_{\ell}(x)\coloneqq \Phi_{\ell}(j(E),x)$ and suppose $\varphi_{\ell}$ has no roots in $\FF_p$ so~$E$ has no $\FF_p$-rational $\ell$-isogenies. Then $\ell$ is not an Elkies prime for~$E$, but it could be that it is an Elkies prime for~$E$ after a small extension of the base field. 
 Schoof proves in~\cite[Proposition~6.2(iii)]{SchoofSEA}  (under our assumptions that~$E$ is ordinary and $j(E)\notin\{0,1728\}$) that if $\varphi_{\ell}$ has no roots in $\FF_p$, it factors into $(\ell+1)/r$ factors of degree $r$, where $r$ is the order of the image of $\pi_E$ in $\PGL(E[\ell])$. The integer $r$ can be computed by computing $\gcd(x^{p^i}-x,\varphi_{\ell}(x))$ for $i=1,2,\ldots,\ell+1$ as is done in Atkin's modification of Schoof's algorithm~\cite[\S\,6]{SchoofSEA}. 
  
 Let $f$ be an irreducible factor of degree $r$ of $\varphi(x)$ and let $\FF_{p^r}=\FF_p[t]/(f(t))$. The division polynomial $\psi_{\ell}$ has a factor  $h$ of degree $ (\ell-1)/2$ over $\FF_{p^r}$ corresponding to the kernel polynomial of a $\FF_{p^r}$-rational $\ell$-isogeny with kernel $K$. 
 It is not  the case that $\pi_E$ induces an endomorphism of $K$, but we can still use the algorithm in~\ref{subsec:mod_ell} to compute $t_{\ell}\coloneqq \tr\pi_E\bmod{\ell}$. 
We may compute $t_{\ell} = \tr \pi_E\bmod{\ell}$ by, 
as before, computing $c$ such that 
\[
    \restr{\pi_E^2}{\ker\phi} +\restr{p}{\ker\phi} = c\restr{\pi_E}{\ker\phi}.
    \]
In the worst case, this is no faster than Schoof's algorithm for computing $t_{\ell}$: the degree of $f$ can be $\ell+1$, and the degree of $h$ will be $(\ell-1)/2$, 
so the polynomial arithmetic in $\FF_{p^r}[x]/(h)$ will be as expensive as the arithmetic in $\FF_p[x]/(\psi_{\ell})$. But it is possible that $\varphi_{\ell}(x)$ can have irreducible factors of degree strictly less than $\ell+1$.

Because~$E$ has a $\FF_{p^r}$-rational $\ell$-isogeny with kernel polynomial $h$, there is a factor of $\psi_{\ell}$ of degree $r(\ell-1)/2$ in $\FF_{p}[x]$ given by $g=\prod_{i=1}^r h^{(p^i)}$. Then the roots of $g$ are the $x$-coordinates of affine points in the orbit of $K$ under $\pi_E$. Let $V$ be the subscheme of~$E$ cut out by $g$. Then  one may compute $t_{\ell}$ by computing 
$c=t_{\ell}$ such that $\restr{\pi_E^2}{V} + \restr{q}{V} = \restr{c\pi_E}{V}$. At a high level, this is the method of Dewaghe~\cite[Section~4]{Dew98} for applying Elkies' method to Atkin primes.

\bibliographystyle{alpha}
\bibliography{bib}

\end{document}